\documentclass[12pt]{article}
\RequirePackage[colorlinks,citecolor=blue,urlcolor=blue,linkcolor=blue]{hyperref}
\hypersetup{
colorlinks = true,
citecolor=blue,
urlcolor=blue,
linkcolor=blue,
pdfauthor = {Alexey Kuznetsov},
pdfkeywords = {Barnes $G$-function, Gamma function, Laplace transform, Pade approximation},
pdftitle = {Computing Barnes $G$-function in the entire complex plane},
pdfpagemode = UseNone
}
\usepackage[titletoc,title]{appendix}
\usepackage{graphicx,xspace,colortbl}
\usepackage{epstopdf}
\usepackage{amsmath,amsthm,amsfonts,amssymb}
\usepackage{color}
\usepackage{enumerate}
\usepackage{fancybox}
\usepackage{epsfig}
\usepackage{subfig}
\usepackage{pdfsync}
\usepackage{comment}
    \def\qed{\hfill$\sqcap\kern-8.0pt\hbox{$\sqcup$}$\\}
    \def\beq{\begin{eqnarray}}
    \def\eeq{\end{eqnarray}}
    \def\beqq{\begin{eqnarray*}}
    \def\eeqq{\end{eqnarray*}}

    \def\re{\textnormal {Re}}
    \def\im{\textnormal {Im}}

    \def\r{{\mathbb R}}

    \def\c{{\mathbb C}}

    \def\d{{\textnormal d}}
    \def\i{{\textnormal i}}

	\newtheorem{lemma}{Lemma}
	\newtheorem{proposition}{Proposition}

\title{Computing the Barnes $G$-function and the gamma function in the entire complex plane}

\author{
Alexey Kuznetsov
\footnote{Dept. of Mathematics and Statistics,  York University,
4700 Keele Street, Toronto, ON, M3J 1P3, Canada. 
 Email:  akuznets@yorku.ca   Webpage: 
\href{https://kuznetsov.mathstats.yorku.ca}{kuznetsov.mathstats.yorku.ca}}
 }

\begin{document}
 
\maketitle

\begin{abstract}
We present an algorithm for generating approximations for the logarithm of Barnes $G$-function in the half-plane $\re(z)\ge 3/2$. These approximations involve only elementary functions and are easy to implement. The algorithm is based on a two-point Pad\'e approximation and we use it to provide two approximations to $\ln(G(z))$, accurate to $3 \times 10^{-16}$ and $3 \times 10^{-31}$ in the half-plane $\re(z)\ge 3/2$; a reflection formula is then used to compute Barnes $G$-function in the entire complex plane. A by-product of our algorithm is that it also produces accurate approximations to the gamma function.
\end{abstract}

\vspace{0.25cm}

{\vskip 0.15cm}
 \noindent {\it Keywords}:  Barnes $G$-function, gamma function, Laplace transform, Pade approximation 
 \\
 \noindent {\it 2010 Mathematics Subject Classification }: Primary 33F05, Secondary 	30E10, 33E30, 33B15


\section{Introduction}\label{section:Introduction}


Barnes $G$-function is an entire function defined via the Weierstrass product
\begin{equation}\label{def:Barnes_G}
G(1+z)=(2\pi)^{z/2} e^{-\frac{1}{2} (z+(1+\gamma)z^2)} 
\prod\limits_{k\ge 1} \Big(1+\frac{z}{k} \Big)^k e^{-z+\frac{z^2}{2k}}, \;\;\; z\in \c. 
\end{equation}
Here $\gamma$ is the Euler-Mascheroni constant. 
This function satisfies
\begin{equation}\label{eq:G_functional_eqn}
G(z+1)=\Gamma(z) G(z).
\end{equation}
and it is known \cite{Vigneras1979} that $G$ is a unique function satisfying \eqref{eq:G_functional_eqn} subject to conditions 
$G(1)=1$ and 
$$
\frac{d^3}{dz^3} \ln G(z)\ge 0 \;\; {\textrm{for }} z>1.
$$

The systematic study of $G(z)$ (and certain related functions, including a more general double gamma function $G(z;\tau)$ and a hierarchy of multiple gamma functions $\Gamma_n(z)$) was initiated by Barnes in a series of papers \cite{Barnes1899,Barnes1901,Barnes1904}. The $G$-function was in fact introduced in  an earlier paper by Alexeiewsky \cite{Alexeiewsky1894} and in some older literature it was sometimes called ``Alexeiewsky function", though by now the name ``Barnes $G$-function" seems to be universally accepted. 

Starting in 1980s the Barnes $G$-function and related double gamma functions started to appear frequently in various areas of Mathematics and Physics. Shintani \cite{Shintani1980} used double gamma function to give a new proof of Kronecker's limit formula. The Barnes $G$-function was found to have connections to determinants of Laplacians  \cite{Sarnak1987,Vardi1988} and it also appeared in the Keating-Snaith 
 conjecture on moments of the Riemann zeta function \cite{Keating2000}.  
 More recently this function has found many aplications in Probability Theory  \cite{Kuznetsov2011,Yor2009,Simon2018}.

Despite this increased interest in the Barnes $G$-function and its apparent usefullness in many applications, no work was done so far on developing efficient numerical algorithms for computing $G(z)$ (though several recent papers studied asymptotic expansions of $G(z)$, see  \cite{Chen2014,Fereira2001,Nemes2014,Xu2017}). It is the purpose of this paper to correct this omission. Our main contribution is an algorithm that can produce accurate and simple approximations to $\ln(G(z))$ in the half-plane $\re(z)\ge 3/2$ -- 
we demonstrate this by providing two such approximations, one valid to $3\times 10^{-16}$ and another to $3 \times 10^{-31}$.
With the help of a reflection formula and the functional equation 
\eqref{eq:G_functional_eqn} this allows one to compute $\ln(G(z))$ in the entire complex plane. 
Our algorithm also generats accurate approximations to the gamma function, though this is probably less important as there already exist excellent approximations to $\Gamma(z)$, such as the Lanczos approximation \cite{Lanczos1964,Pugh2004}.
 
The key part of our algorithm consists in approximating the function 
\begin{equation}\label{def:f(x)}
f(x):=\frac{e^{-x}}{x^3} \left[\frac{1}{2} \coth\left(\frac{x}{2}\right)- \frac{1}{x}-\frac{x}{12} \right], \;\;\; x>0,
\end{equation}
by a finite sum of exponential functions $\sum_{j=1}^p c_j e^{-\lambda_j x}$, and this is done by finding a two-point Pad\'e approximation to the Laplace transform of $f(x)$. This two-point Pad\'e approximation method is of independent interest, as it can be used to provide approximations (by sums of exponential functions) to any function $g$ for which derivatives $g^{(n)}(0)$ and ``moments" $\int_0^{\infty} x^n g(x) \d x$ are known explicitly or can be computed to sufficiently high precision. 

The paper is organized as follows. In Section \ref{Section2} we explain how approximations to $f(x)$ by a finite sum of exponential functions lead to approximations to $\ln(\Gamma(z))$ and $\ln(G(z))$ in the half-plane $\re(z)\ge 3/2$ and how one can extend these approximations to the entire complex plane by using functional equations and reflection formulas. Section \ref{Section3} describes the two-point Pad\'e approximation algorithm for finding exponential sum approximations and in Section \ref{Section4} we present numerical results that illustrate the performance of this algorithm. In Appendix \ref{AppendixA} we describe an algorithm for computing the dilogarithm function (this is needed for implementing the reflection formula for the Barnes $G$-function), and in Appendix \ref{AppendixB} we present the numerical values of the coefficients $c_j$ and $\lambda_j$.

\section{Approximatng logarithms of $\Gamma(z)$ and $G(z)$}\label{Section2}

The starting point for our investigation is the following proposition, which reduces the computation of the gamma function and the Barnes $G$-function to the computation of 
\begin{equation}\label{def_Fz}
F(z):=\int\limits_0^{\infty} e^{-zx } x f(x)  \d x
\end{equation}
and its derivative $F'(z)$. In what follows, $A$ denotes the the Glaisher-Kinkelin constant and $\ln(z)$ denotes the principal branch of the logarithm function. 

\begin{proposition}\label{prop1} For $\re(z)>0$ 
\begin{align}
\label{eq:ln_Gamma}
\ln(\Gamma(z))&=\left(z-\frac12 \right)\ln(z) - z+\frac12 \ln(2\pi)  + \frac{1}{12z}-F'(z-1), \\
\label{eq:ln_G}
\ln(G(z))&=\left(\frac{z^2}{2}-z+\frac{5}{12}\right)\ln(z)
-\frac{3}{4}z^2
+\frac{1}{2}\ln(2\pi)(z-1)+z \\ \nonumber
&+\frac{1}{12}-\ln(A)-\frac{1}{12z}+F(z-1)-(z-1)F'(z-1).
\end{align}
\end{proposition}
\begin{proof}
Formula \eqref{eq:ln_Gamma} is obtained from the well-known Binet's integral representation 
$$
\ln(\Gamma(z))=\left(z-\frac{1}{2}\right)\ln(z)-z+\frac{1}{2} \ln(2\pi) + \int_0^{\infty} \frac{e^{-zx}}{x} \left[\frac{1}{e^{x}-1}
-\frac{1}{x}+\frac{1}{2}\right] \d x, 
$$
 by noting that $f$ can be expressed in an equivalent form
\begin{equation}\label{eq:f(x)_2nd}
f(x)=
\frac{e^{-x}}{x^3} \left[\frac{1}{e^x-1} -\frac{1}{x} +\frac{1}{2}-\frac{x}{12} \right].
\end{equation}
Next, by setting $N=1$ in
\cite{Fereira2001}[Equation (12)] we obtain 
$$
\ln(G(z))=(z-1)\ln(\Gamma(z)) - \frac12 \left(z^2-z+\frac16\right)\ln(z)
+\frac{z^2}4-\ln(A)+F(z-1), 
$$
and this result, together with \eqref{eq:ln_Gamma}, implies \eqref{eq:ln_G}. 
\end{proof}

Next we describe our approximation algorithm. 

\subsection*{Approximations in the domain $\re(z)\ge 3/2$}

Let $p$ be a positive integer. Assume we have found  
 coefficients $\{c_j\}_{1\le j \le p}$ and 
$\{\lambda_j\}_{1\le j \le p}$  such that $\re(\lambda_j)>0$ and  
\begin{equation}\label{def_phi}
\phi(x):=\sum\limits_{j=1}^p c_j e^{-\lambda_j x}
\end{equation}
 is a ``good" approximation to $f(x)$ on $x\in (0,\infty)$, in the sense that the numbers
\begin{align}\label{eq:f_sum_error1}
 \epsilon_1&:= 2\sup\limits_{x>0}  \Big \vert x^2( f(x)-\phi(x)) \Big \vert, \\
 \label{eq:f_sum_error2}
  \epsilon_2&:= 2\sup\limits_{x>0} 
  \Big \vert 3x(f(x)-\phi(x))+x^2(f'(x)- \phi'(x)) \Big \vert, 
 \end{align}
are small. 
For $\re(z)>0$ we define a rational function
 \begin{equation}\label{def:Phi}
 \Phi(z):=
  \sum\limits_{j=1}^p \frac{c_j}{(z+\lambda_j)^{2}} =\int_0^{\infty} e^{-z x} x \phi(x) \d x
 \end{equation}
and we define approximations to $\ln(\Gamma(z))$ and $\ln(G(z))$ in the form
\begin{align}
\label{eq:ln_Gamma_approx}
\ln(\widehat \Gamma(z))&=\left(z-\frac12 \right)\ln(z) - z+\frac12 \ln(2\pi)  + \frac{1}{12z}-\Phi'(z-1), \\
\label{eq:ln_G_approx}
\ln(\widehat G(z))&=\left(\frac{z^2}{2}-z+\frac{5}{12}\right)\ln(z)
-\frac{3}{4}z^2
+\frac{1}{2}\ln(2\pi)(z-1)+z \\ \nonumber
&+\frac{1}{12}-\ln(A)-\frac{1}{12z}+\Phi(z-1)-(z-1)\Phi'(z-1).
\end{align}

 \begin{proposition}
 In the half-plane $\re(z)\ge 3/2$ we have 
 \begin{equation}\label{ln_gamma_estimate}
 \vert \ln(\Gamma(z))-\ln(\widehat \Gamma(z)) \vert \le  \epsilon_1
 \end{equation}
 and 
  \begin{equation}\label{ln_G_estimate}
 \vert \ln(G(z))-\ln(\widehat G(z)) \vert \le  \epsilon_2. 
 \end{equation}
 \end{proposition}
 \begin{proof} 
 From \eqref{eq:ln_Gamma} and \eqref{eq:ln_Gamma_approx} we see that 
 $$
 \ln(\Gamma(z))-\ln(\widehat \Gamma(z))=\Phi'(z-1)-F'(z-1).
 $$
 We set $w=z-1$ and estimate for $\re(w) \ge 1/2$  
 $$
 \vert F'(w)- \Phi'(w)\vert = \Big \vert \int_0^{\infty} e^{-wx} x^2 (f(x)-\phi(x))\d x \Big \vert 
 \le \int_0^{\infty} e^{-x/2} x^2 \vert f(x)-\phi(x)\vert \d x \le  
 \frac{\epsilon_1}{2} \int_0^{\infty} e^{-x/2}\d x=\epsilon_1,
 $$
 which proves \eqref{ln_gamma_estimate}. Similarly, from \eqref{eq:ln_G} and \eqref{eq:ln_G_approx} we see that 
 $$
 \ln(G(z))-\ln(\widehat G(z)) = F(z-1)-(z-1) F'(z-1) - \Phi(z-1) + (z-1) \Phi'(z-1).
 $$ 
 We perform integration by parts 
 $$
 w F'(w)=-w\int_0^{\infty} e^{-w x} x^2 f(x) \d x=-\int_0^{\infty} e^{- w x}(2 x f(x)+x^2 f'(x)) \d x
 $$
 and estimate for $\re(w)\ge 1/2$ 
 \begin{align*}
 \vert F(w)- w F'(w) - \Phi(w) + w \Phi'(w) \vert 
 & = \Big \vert 
 \int_0^{\infty} e^{-w x} (3 x (f(x)-\phi(x))+x^2 (f'(x)-\phi'(x)) ) \d x \Big \vert \\
 &\le \int_0^{\infty} e^{-x/2}  \vert 3 x (f(x)-\phi(x))+x^2 (f'(x)-\phi'(x)) \vert  \d x \\
 & \le \frac{ \epsilon_2}{2} \int_0^{\infty} e^{-x/2}\d x=\epsilon_2. 
 \end{align*}
 \end{proof}

\subsection*{Extending the domain to $\re(z)\ge 1/2$}

 In the vertical strip $1/2 \le \re(z) < 3/2$ we compute $\ln(\Gamma(z))$ 
 via  functional equation 
 $$
 \ln(\Gamma(z))=\ln(\Gamma(z+1))-\ln(z),
$$ 
  and then we compute 
 $\ln(G(z))$ via  
$$
 \ln(G(z))=\ln(G(z+1))-\ln(\Gamma(z)).
 $$

 \subsection*{Approximations in the domain $\re(z)<1/2$: applying reflection formulas} 
 
 It is well-known that Gamma function satisfies the reflection identity 
 \begin{equation}\label{eqn:Gamma_reflection1}
 \Gamma(z) \Gamma(1-z) = \frac{\pi}{\sin(\pi z)}. 
 \end{equation}
 In the previous sections we described an algorithm for computing $\ln(\Gamma(z))$ for $\re(z)\ge 1/2$ and the above identity gives a simple algorithm for computing $\Gamma(z)$ for $\re(z)<1/2$. However, some additional work is needed if we want to compute $\ln(\Gamma(z))$ for $\re(z)<1/2$ while preserving analyticity of this function in $\c \setminus (-\infty, 0]$. First of all, we rewrite the above reflection identity in the form
 $$
 \Gamma(z)=\frac{1}{\Gamma(1-z)} \times \frac{-2\pi \i e^{\pi \i z}}{1-e^{2\pi \i z}}. 
 $$
 We assume that $\im(z)>0$ and take logarithms of both sides of the above equation and  obtain 
 \begin{equation}\label{eqn:Gamma_reflection2}
 \ln(\Gamma(z))=
 -\ln(\Gamma(1-z))+\ln(2\pi)-\frac{\pi \i}{2}+\pi \i z- \ln(1-e^{2\pi \i z}).
 \end{equation}
 One can check that we chose the correct branch of logarithm by noting that the above equation is valid if $\re(z)=1/2$ and $\im(z) \to 0^+$. The right-hand side of \eqref{eqn:Gamma_reflection2} is analytic in $\im(z)>0$, thus we can use this identity to compute $\ln(\Gamma(z))$ for $\re(z)<1/2$ and $\im(z)>0$. This identity is also used for $0<z<1/2$ (again, we need to use the principal branch of the logarithm).  When $\re(z)<1/2$ and $\im(z)<0$ we use the above method coupled with the identity $\overline{\ln(\Gamma(z))}=\ln(\Gamma(\bar z))$. 
 
Computing $\ln(G(z))$ in the half-plane $\re(z)<1/2$ follows a similar procedure. 
In this case the relevant reflection formula is 
$$
\frac{G(1+z)}{G(1-z)} = (2\pi)^z \exp\Big( - \int_0^z \pi t \cot(\pi t) \d t \Big).
$$
Using formulas (4.6), (8.43) in \cite{Lewin} the integral in the right-hand side of the above identity can be expressed in terms of the dilogarithm function 
\begin{equation}\label{eq:def_Li2}
{\textrm{Li}}_2(z):=-\int_0^z \ln(1-t) \frac{\d t}{t}=\sum\limits_{k\ge 1} \frac{z^2}{k^2}, \;\;\; |z|\le 1, 
\end{equation} 
so that changing the variable $z\mapsto z-1$ we would obtain a reflection formula in the form 
\begin{align}\label{eq:lnG_reflection1}
\ln(G(z))&=
\ln(G(2-z))+(z-1) \ln(2\pi) +\frac{\pi \i}{2} (z^2-2z+5/6) 
\\ \nonumber & - (z-1) \ln(1-e^{2\pi \i z})- \frac{1}{2\pi \i} {\textrm{Li}}_2(e^{2\pi \i z}), \;\;\; \im(z)>0.
\end{align}
Again, one can check that we have the correct branch of the logarithm by taking $\re(z)=1$ and $\im(z) \to 0^+$ and using the fact that $G(1)=1$ and  
${\textrm{Li}}_2(1)=\pi^2/6$. 
We note that the right-hand side of \eqref{eq:lnG_reflection1} is analytic in the half-plane $\im(z)>0$, thus we can use \eqref{eq:lnG_reflection1} 
to compute $\ln(G(z))$ for $\re(z)<1/2$ and $\im(z)>0$. An efficient algorithm for computing the dilogarithm function ${\textrm{Li}}_2(z)$ for $|z|\le 1$ is presented in the Appendix \ref{AppendixA}. We also use formula \eqref{eq:lnG_reflection1} when $0<z<1/2$. When $\re(z)<1/2$ and $\im(z)<0$ we use the above method coupled with the identity $\overline{\ln(G(z))}=\ln(G(\bar z))$.

\vspace{0.25cm}

This concludes our description of the method for computing $\ln(\Gamma(z))$ and $\ln(G(z))$ in $\c\setminus (-\infty, 0]$. The only outstanding question is how to approximate the function $f$  defined in \eqref{def:f(x)} by a finite sum of exponential functions of the form \eqref{def_phi} so that the errors $\epsilon_1$ and $\epsilon_2$ in \eqref{eq:f_sum_error1} and \eqref{eq:f_sum_error2} are small. We propose a solution to this problem in the next section.

\section{Approximating a function by a finite sum of exponential functions}\label{Section3}

There exist a number of methods for approximating a given function by a finite sum of exponential functions. In \cite{Rice1962} the problem of finding the best approximation on a finite interval in the Chebyshev sense is considered. A method based on Hankel matrices is introduced in \cite{Beylkin2005} and a Poisson summation approach is presented in \cite{Beylkin2010}. See also \cite{Braess2005} for results on approximating $1/x$ on $[1,\infty)$. In this paper we propose a different approach for approximating functions by  finite sums of exponential functions: this method is based on a two-point Pad\'e approximation and it is particularly well-suited for the function $f$ given in \eqref{def:f(x)}. 

First we show that for the function $f$ defined by \eqref{def:f(x)} one can compute explicitly (in terms of values of the Riemann zeta function $\zeta(s)$ at positive integers and some other well-known constants) both the derivatives at zero and the moments.  

\begin{lemma}\label{Lemma1} ${}$
\begin{itemize}
\item[(i)] For $n\ge 0$ 
 \begin{equation}\label{eq:f_derivatives}
 f^{(n)}(0)=(-1)^n n! \sum\limits_{k=0}^{[n/2]} \frac{B_{2k+4}}{(n-2k)!(2k+4)!}, 
 \end{equation}
 where $B_n$ are Bernoulli numbers. 
 \item[(ii)] Let $\mu_n:=\int_0^{\infty} x^n f(x) \d x$. Then  
 \begin{align*}
\mu_0&=\frac{1}{2} \zeta'(-2)+\frac{1}{72}=-\frac{\zeta(3)}{8 \pi^2}+\frac{1}{72}, \\
\mu_1&=-\zeta'(-1)-\frac{1}{6}=\ln(A)-\frac{1}{4}, \\
\mu_2&=\zeta'(0)+\frac{11}{12}=-\ln(2\pi)/2+\frac{11}{12}, \\
\mu_3&=\lim\limits_{s\to 1} \Big(\zeta(s)-\frac{1}{s-1} \Big)-\frac{7}{12}=\gamma-7/12, 
\end{align*} 
and for $n\ge 4$
\begin{equation*}
\mu_n=(n-3)! \Big[ \zeta(n-2)-\frac{1}{n-3}-\frac{1}{2}
-\frac{n-2}{12} \Big].
\end{equation*}
Here $\gamma$ is the Euler-Mascheroni constant and $A$ is the Glaisher-Kinkelin constant.
\end{itemize}
\end{lemma}
\begin{proof}
Formula \eqref{eq:f_derivatives} follows from the Taylor series expansion of the hyperbolic cotangent function: 
$$
 \frac{1}{2} \coth\left(\frac{x}{2}\right)=
 \sum\limits_{n=0}^{\infty} \frac{B_{2n}}{(2n)!} x^{2n-1}. 
 $$
Formulas for moments $\mu_n$ follow from the identity  
\begin{equation}\label{eq:f_Mellin}
\int_0^{\infty} x^{s+2} f(x) \d x=
\Gamma(s) \Big[ \zeta(s)-\frac{1}{s-1}-\frac{1}{2}
-\frac{s}{12} \Big],\;\;\; \re(s)>-3,
\end{equation} 
which can be obtained from \eqref{eq:f(x)_2nd} and integral representations 
\begin{align*}
&\int_0^{\infty} \frac{e^{-x} x^{s-1}}{e^x-1} \d x=\Gamma(s)(\zeta(s)-1), \;\;\; \re(s)>1, \\
&\int_0^{\infty} e^{-x} x^{s-1} \d x=\Gamma(s), \;\;\; \re(s)>0. 
\end{align*}
The above two identities can be used to deduce \eqref{eq:f_Mellin} in the domain $\re(s)>1$ and extension of \eqref{eq:f_Mellin} to the half-plane $\re(s)>-3$ is done by analytic continuation (note that both sides of \eqref{eq:f_Mellin} are well-defined and analytic in $\re(s)>-3$). Now we compute 
\begin{align*}
\mu_0=\int_0^{\infty} f(x) \d x&=\lim\limits_{s\to -2} 
\Gamma(s) \Big[ \zeta(s)-\frac{1}{s-1}-\frac{1}{2}
-\frac{s}{12} \Big]\\
&=\lim\limits_{s\to -2} 
(s-2) \Gamma(s) \times \frac{\zeta(s)-\frac{1}{s-1}-\frac{1}{2}
-\frac{s}{12}}{s-2}\\
&=\frac{1}{2} \frac{\d }{\d s} \Big[ \zeta(s)-\frac{1}{s-1}-\frac{1}{2}
-\frac{s}{12} \Big] \bigg \vert_{s=-2}=\frac{1}{2} \zeta'(-2)+\frac{1}{72},
\end{align*}
and the second expression for $\mu_0$ is obtained from the reflection formula: 
$\zeta'(-2)=-\zeta(3)/(4\pi^2)$. The proof of formulas for $\mu_1$, $\mu_2$ and $\mu_3$ is identical and we leave all details to the reader. Formulas for $\mu_n$ with $n\ge 4$ follow directly from \eqref{eq:f_Mellin}. 
\end{proof}

Now we describe our method for finding approximations to $f(x)$ in the form
$\phi(x)=\sum\limits_{j=1}^p c_j e^{-\lambda_j x}$. We fix two non-negative integers $N_1$ and $N_2$ such that $N_1+N_2=2p$ and we aim to find coefficients 
$\{c_j\}_{1\le j \le p}$ and $\{\lambda_j\}_{1\le j \le p}$ such that the functions $f$ and $\phi$ have the  same first $N_1$ moments and the same first $N_2$ terms in Taylor series expansion at zero. In other words, we want to enforce the following  conditions: 
\begin{align}\label{eq:tilde_f_moments}
&\int_0^{\infty} x^n \phi(x) \d x=
n! \sum\limits_{j=1}^p c_j \lambda_j^{-n-1}=\mu_n, \;\;\; n=0,1,2,\dots,N_1-1, \\
\label{eq:tilde_f_derivatives}
&\phi^{(n)}(0)=(-1)^n \sum\limits_{j=1}^p c_j \lambda_j^n=f^{(n)}(0), \;\;\;
n=0,1,2,\dots,N_2-1.
\end{align}

As we will show next, this problem is equivalent to a two-point Pad\'e approximation, with the two points being $z=0$ and $z=\infty$. Indeed, let us consider the Laplace transform 
\begin{equation}\label{def:tilde_F}
\Psi(z):=\int_0^{\infty} e^{-zx} \phi(x) \d x=\sum\limits_{j=1}^p \frac{c_j}{z+\lambda_j}. 
\end{equation}
From the above partial fraction decomposition it is easy to obtain the following series representations
\begin{align*}
\Psi(z)&=\sum\limits_{n\ge 0} z^{-n-1} \Big[ (-1)^n \sum_{j=1}^p c_j \lambda_j^n \Big], \\
\Psi(z)&=\sum\limits_{n\ge 0} z^{n} \Big[ (-1)^n \sum_{j=1}^p c_j \lambda_j^{-n-1} \Big],
\end{align*}
where the first one converges for all $z$ large enough and the second for all $z$ in a neighbourhood of zero. In light of the above equations we see that the problem of finding a function $\phi(x)=\sum\limits_{j=1}^p c_j e^{-\lambda_j x}$ that satisfies conditions 
\eqref{eq:tilde_f_moments} and \eqref{eq:tilde_f_derivatives} is equivalent to the following problem: find the coefficients $\{a_j\}_{1\le j \le p}$ and 
$\{b_j\}_{1\le j \le p}$ of the rational function 
\begin{equation}\label{eq:tilde_F_rational}
\Psi(z)=\frac{P(z)}{Q(z)}=\frac{a_0+a_1z+\dots+a_{p-1} z^{p-1}}{b_0+b_1z+\dots+b_{p-1}z^{p-1} +z^p} 
\end{equation}
such that 
\begin{align}\label{F_asymptotics_0}
\Psi(z)&=\beta_0+\beta_1 z +\beta_2 z^2 + \dots + \beta_{N_1-1} z^{N_1-1} + O(z^{N_1}), \;\;\; z\to 0,\\
\label{F_asymptotics_infty}
\Psi(z)&=\alpha_0 z^{-1} + \alpha_1 z^{-2} + \dots \alpha_{N_2-1} z^{-N_2} + O(z^{-N_2-1}), \;\;\; z\to \infty,
\end{align}
where we denoted
\begin{equation}\label{eq:alpha_and_beta}
\alpha_n:=f^{(n)}(0), \;\;\; \beta_n:=\frac{(-1)^n}{n!} \mu_n.
\end{equation} 
This is precisely the two-point Pad\'e approximation problem. 
As in the case of the classical (one-point) Pad\'e approximation, the above problem can be reduced to solving a system of linear equations. 
First we rewrite equation \eqref{F_asymptotics_0} in the form
\begin{align*}
a_0+a_1 z + \dots + a_{p-1} z^{p-1} &= \Big( b_0+b_1z+\dots+b_{p-1}z^{p-1} +z^p \Big) 
\\ &\times \Big( \beta_0+\beta_1 z +\beta_2 z^2 + \dots + \beta_{N_1-1} z^{N_1-1} + O(z^{N_1})\Big).
\end{align*}
Comparing the coefficients in front of $z^k$ with $k=0,1,\dots,N_1-1$ we obtain equations 
\begin{equation}\label{linear_equations1}
a_k {\mathbf 1}_{\{k\le p-1\}}=\beta_{k-p} {\mathbf 1}_{\{p\le k\}} + \sum_{j=0}^{p-1} 
 b_j \beta_{k-j} {\mathbf 1}_{\{0\le k-j \le N_1-1\}}. 
 \end{equation}
Next we introduce a variable $u=1/z$ and rewrite \eqref{F_asymptotics_infty} in the form
\begin{align*}
a_{p-1}+a_{p-2}u+\dots+a_0 u^{p-1} &= \Big( 1+b_{p-1} u + \dots + b_0 u^p \Big) \\
& \times \Big( \alpha_0 + \alpha_1 u + \dots + \alpha_{N_2-1} u^{N_2-1} + O(u^{N_2}) \Big). 
\end{align*}
Comparing the coefficients in front of $u^k$ with $k=0,1,\dots,N_2-1$ we obtain equations
\begin{equation}\label{linear_equations2}
a_{p-1-k} {\mathbf 1}_{\{k \le p-1\}} = \alpha_k  + \sum\limits_{j=0}^{p-1} b_{j} \alpha_{k+j-p} 
{\mathbf 1}_{\{0 \le k+j-p \le N_2-1\}}. 
\end{equation}
Now we have $N_1$ linear equations of the form \eqref{linear_equations1} and $N_2$ linear equations of the form \eqref{linear_equations2} -- together these give us $N_1+N_2=2p$ equations for $2p$ unknowns $\{a_j, b_j\}_{0\le j \le p-1}$. To rewrite these equations in a matrix form (that is more suitable for numerical computation), 
we introduce a $2p \times 2p$ matrix ${\bf A}=\{A_{i,j}\}_{1\le i,j \le 2p}$ and a column vector $\vec{\bf y}\in \r^{2p}$ by specifying their elements as follows: for $k=0,1,\dots,N_1-1$ and $j=0,1,\dots,p-1$ 
\begin{align*}
A_{k+1,k+1}&={\mathbf 1}_{\{k\le p-1\}},\\
A_{k+1,p+1+j}&= -\beta_{k-j} {\mathbf 1}_{\{0 \le k-j \le N_1-1\}},\\
y_{k+1}&=\beta_{k-p} {\mathbf 1}_{\{p \le k \}}, 
\end{align*}
and for $l=0,1,\dots,N_2-1$ and $j=0,1,\dots,p-1$ 
\begin{align*}
A_{N_1+l+1,p-l}&={\mathbf 1}_{\{l\le p-1\}},\\
A_{N_1+l+1,p+1+j}&= -\alpha_{l+j-p} {\mathbf 1}_{\{0 \le l+j-p \le N_2-1\}},\\
y_{N_1+l+1}&=\alpha_l.  
\end{align*}
Now the two sets of linear equations \eqref{linear_equations1}  and
\eqref{linear_equations2} can be represented as a single matrix equation ${\bf A} \vec{\bf x}=\vec{\bf y}$, where 
a column vector $\vec{\bf x}\in \r^{2p}$ is defined as 
\begin{equation}
\vec{\bf x}:=[a_0,a_1,\dots,a_{p-1},b_0,b_1,\dots,b_{p-1}]^T. 
\end{equation}
After we solve equation ${\bf A} \vec{\bf x}=\vec{\bf y}$ and find $\vec{\bf x}$, we compute $a_j=x_{j+1}$ and $b_j=x_{p+j+1}$ for $j=0,1,\dots,p-1$, and this gives us the desired rational function $\Psi$ in the form  \eqref{eq:tilde_F_rational}. Performing the partial fraction decomposition (and assuming that all the roots of the denominator $Q(z)$ are simple) we would obtain the expression of $\phi$ in the form 
\eqref{def:tilde_F} and this would give us the desired coefficients 
 $\{c_j\}_{1\le j \le p}$ and $\{\lambda_j\}_{1\le j \le p}$ 
of $\phi(x)=\sum\limits_{j=1}^p c_j e^{-\lambda_j x}$. We can check the accuracy of computation of these coefficients via \eqref{eq:tilde_f_moments} and \eqref{eq:tilde_f_derivatives}.

\vspace{0.25cm}
\noindent
{\bf Remark 1:}  
Note that the functions $\Phi$ and $\Psi$ defined in \eqref{def:Phi} and \eqref{def:tilde_F} satisfy $\Phi(z)=-\Psi'(z)$. Thus the rational function $\Phi$ has the same $N_1-1$ first terms of Taylor expansion at zero and the same first $N_2+1$ asymptotic terms as $z\to +\infty$ as the function $F(z)$ . Therefore our approximations $\ln(\widehat \Gamma(z))$ and $\ln(\widehat G(z))$, which were introduced in \eqref{ln_gamma_estimate} and \eqref{ln_G_estimate}, satisfy the following properties: 
as $z\to 1^+$ 
\begin{align*}
&\ln(\Gamma(z))-\ln(\widehat \Gamma(z))=O((z-1)^{N_1-1}), \\
&\ln(G(z))-\ln(\widehat G(z))=O((z-1)^{N_1-1}), 
\end{align*} 
and as $z\to \infty$ 
\begin{align*}
&\ln(\Gamma(z))-\ln(\widehat \Gamma(z))=O(z^{-N_2-2}), \\
&\ln(G(z))-\ln(\widehat G(z))=O(z^{-N_2-2}).  
\end{align*} 
The first of these is probably not very informative, as we only use approximations 
$\ln(\widehat \Gamma(z))$ and $\ln(\widehat G(z))$ in the half-plane $\re(z)\ge 3/2$ (where $z$ is bounded away from $1$), but the second result is useful as it tells us that our approximations in the half-plane $\re(z) \ge 3/2$ become more accurate as $|z|$ increases.  
\vspace{0.25cm}

To concluse this section, we would like to point out that there is no theoretical justification for the above algorithm. One can not be certain that the system  ${\bf A} \vec{\bf x}=\vec{\bf y}$ has a solution or that all roots of polynomial $Q(z)$ in \eqref{eq:tilde_F_rational} are simple (though in all our numerical experiments this was always the case). Even if the algorithm does produce the numbers $\{c_j\}_{1\le j \le p}$ and $\{\lambda_j\}_{1\le j \le p}$, there is no guarantee that this would result in a good approximation 
$\phi(x)=\sum\limits_{j=1}^p c_j e^{-\lambda_j x}$ to $f(x)$ (in the sense of errors 
$\epsilon_1$ and $\epsilon_2$ in \eqref{eq:f_sum_error1} and \eqref{eq:f_sum_error2} being small). Existing results on convergence of Pad\'e approximations to Stieltjes functions are not applicable here since the function $f(x)$ is not completely monotone (thus its Laplace transform is not a Stieltjes function). Lacking any theoretical justification for the above algorithm, in the next section we will try to convince the reader that it works well by presenting numerical evidence.

\section{Numerical results}\label{Section4}

First we would like to make several remarks about implementation of the algorithm presented in Section \ref{Section3}. The algorithm includes two steps that are ill-conditioned and that require high precision arithmetic: (i) solving the system of linear equations ${\bf A} \vec{\bf x}=\vec{\bf y}$ and (ii) finding the partial fraction decomposition of the rational function $\Psi(z)$ (which is equivalent to finding roots of the polynomial $Q(z)$ in \eqref{eq:tilde_F_rational}). Because of this issue, we perform all computations with the precision of 500 digits (we use David Bailey's Fortran module MPFUN).  The inputs of the algorithm are the constants $f^{(n)}(0)$ and $\mu_n$ given in Lemma \ref{Lemma1}: we pre-compute them also to the precision of 500 digits. We used existing tables of Bernoulli numbers (which also give us the values of $\zeta(2n)$) and we computed the values of the $\zeta(2n+1)$ via Euler-Maclaurin summation.

\begin{figure}[t!]
\centering
\subfloat[The graph of $f(x)$]{\label{fig1_p1}\includegraphics[height =6cm]{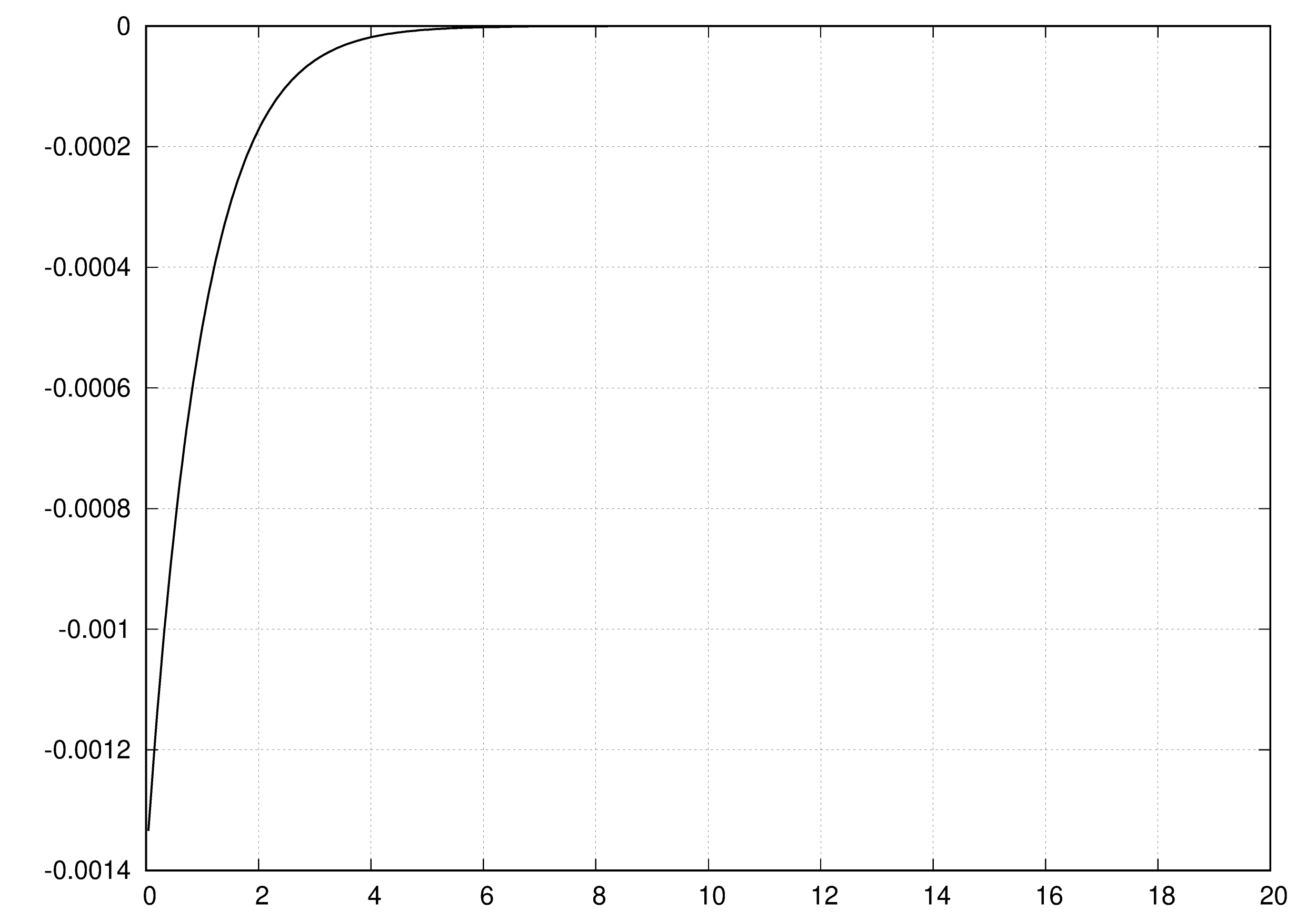}} 
\qquad
\subfloat[The graph of $f(x)-\sum\limits_{j=1}^3 c_j e^{-\lambda_j x}$]{\label{fig1_p2}\includegraphics[height =6cm]{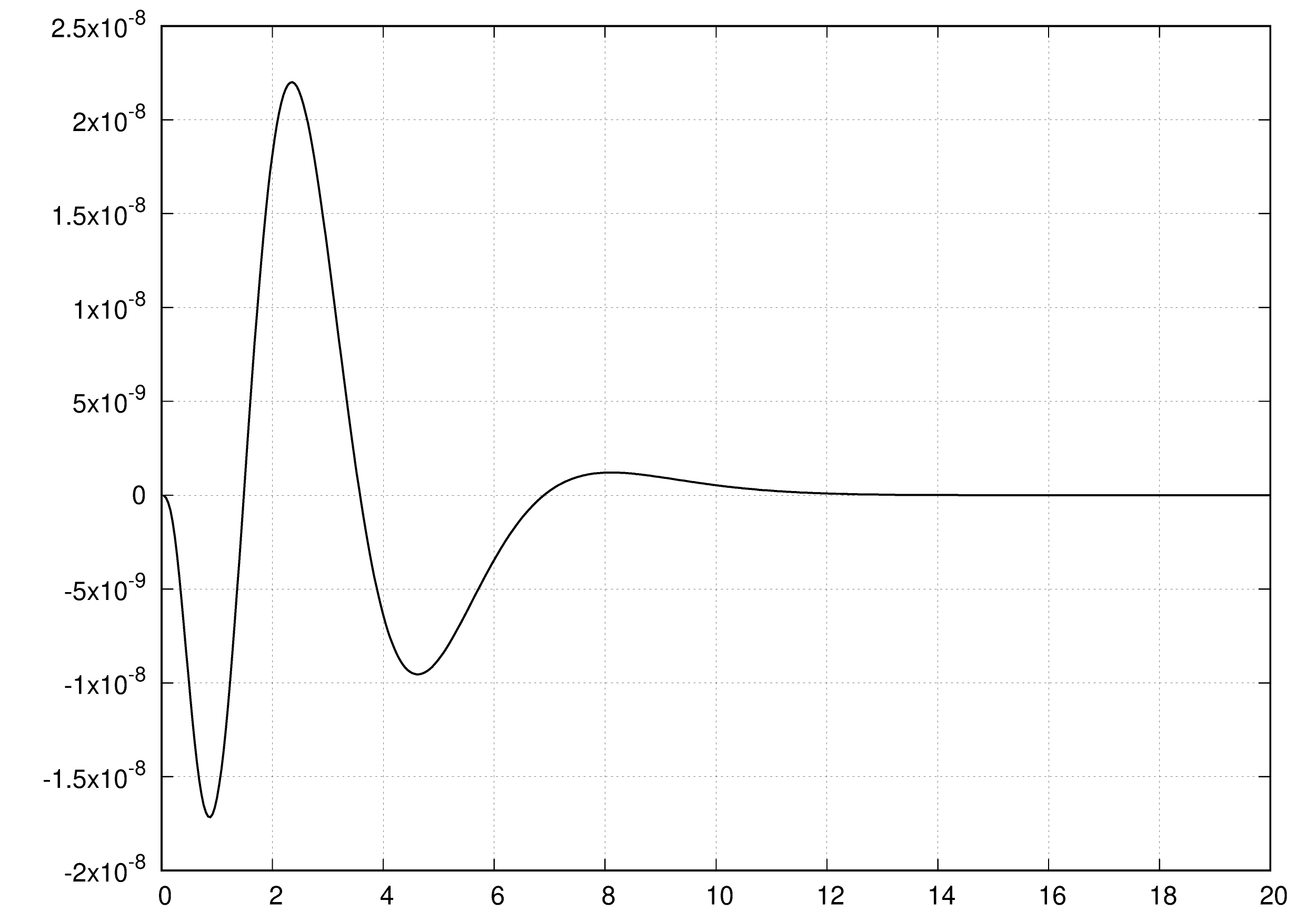}}
\caption{Here $\{\lambda_j,c_j\}_{1\le j \le 3}$ are obtained by a two-point Pad\'e approximation algorithm with $p=N_1=N_2=3$.} 
\label{fig1}
\end{figure}

First we do a test run of our approximation algorithm with $p=N_1=N_2=3$. The two-point Pad\'e approximation algorithm produces the following values of $\{\lambda_i\}_{1\le i \le 3}$ and $\{c_i\}_{1\le i \le 3}$:  
\begin{align*}
\lambda_1&=1.119578657398, \\
\lambda_2&=1.348997960953+ \i \times 0.462842303653, \\
c_1&=-1.501034759773\times 10^{-3}, \\
 c_2&=5.607293544244\times 10^{-5} - \i \times 1.516208003305\times 10^{-4},  \\
 \lambda_3&={\overline{\lambda_2}}, \;\;\; c_3={\overline{c_2}}.
\end{align*}
One can check (numerically) that these numbers indeed satisfy conditions \eqref{eq:tilde_f_moments} and \eqref{eq:tilde_f_derivatives}, thus we have found 
a function $\phi(x)=\sum\limits_{j=1}^3 c_j e^{-\lambda_j x}$ which has the same first three coefficients of Taylor expansion at zero as $f(x)$ and the same first three moments as $f(x)$. 
We present the graphs of $f(x)$ and of the difference $f(x)-\phi(x)$ on Figure \ref{fig1}. The results already look promising: we can approximate $f(x)$ to accuracy of $2.2 \times 10^{-8}$ with a sum of just three exponential functions.

\begin{table}[t!]
\centering
 \begin{tabular}{|c || c c c c c c |} 
 \hline
  & $N_1=0$ & $N_1=1$ & $N_1=2$ & $N_1=3$ & $N_1=4$ & $N_1=5$ \\ [0.5ex] 
 \hline\hline
 $N_2=0$ & & & $4.3\times 10^{-5}$ &  & $4.0 \times 10^{-6}$ &  \\
 $N_2=1$ & & $1.2\times 10^{-5}$
 &  & $6.1\times 10^{-7}$  & & $9.7 \times 10^{-9}$ \\ 
 $N_2=2$ & $1.6\times 10^{-5}$ &   & $4.5 \times 10^{-7}$ & & $1.5 \times 10^{-8}$  & \\
 $N_2=3$ & & $7.2\times 10^{-7}$
  &   & $2.2 \times 10^{-8}$ & & $7.4 \times 10^{-9}$ \\
 $N_2=4$ & $2.3 \times 10^{-6}$ &   & $3.7\times 10^{-8}$ &  & $3.6\times 10^{-9}$ & \\
 $N_2=5$ &  & $1.1 \times 10^{-7}$ &   & $4.2\times 10^{-9}$ & &  $3.5 \times 10^{-10}$ \\ [1ex] 
 \hline
 \end{tabular}
  \caption{The error or approximation $\sup\limits_{x\ge 0} \vert f(x)-\sum\limits_{j=1}^p c_j e^{-\lambda_j x} \vert$ for various values of $N_1$ and $N_2$.}
\label{table:1}
\end{table}

Next we performed a number of similar experiments with all values of $N_1$ and $N_2$ between zero and five (restricted by the requirement that $p=(N_1+N_2)/2$ must be an integer). In each case we computed numerically the $L_{\infty}$ norm 
$$
\sup\limits_{x\ge 0}  \vert f(x)-\sum\limits_{j=1}^p c_j e^{-\lambda_j x} \vert. 
$$
The results of these computations can be seen in Table \ref{table:1}. We see that the accuracy of approximation improves with increasing $p$ or $N_1$ or $N_2$.

\begin{figure}[t!]
\centering
\subfloat[$N_1=22$, $N_2=8$ and $p=15$]{\label{fig2_p1}\includegraphics[height =6cm]{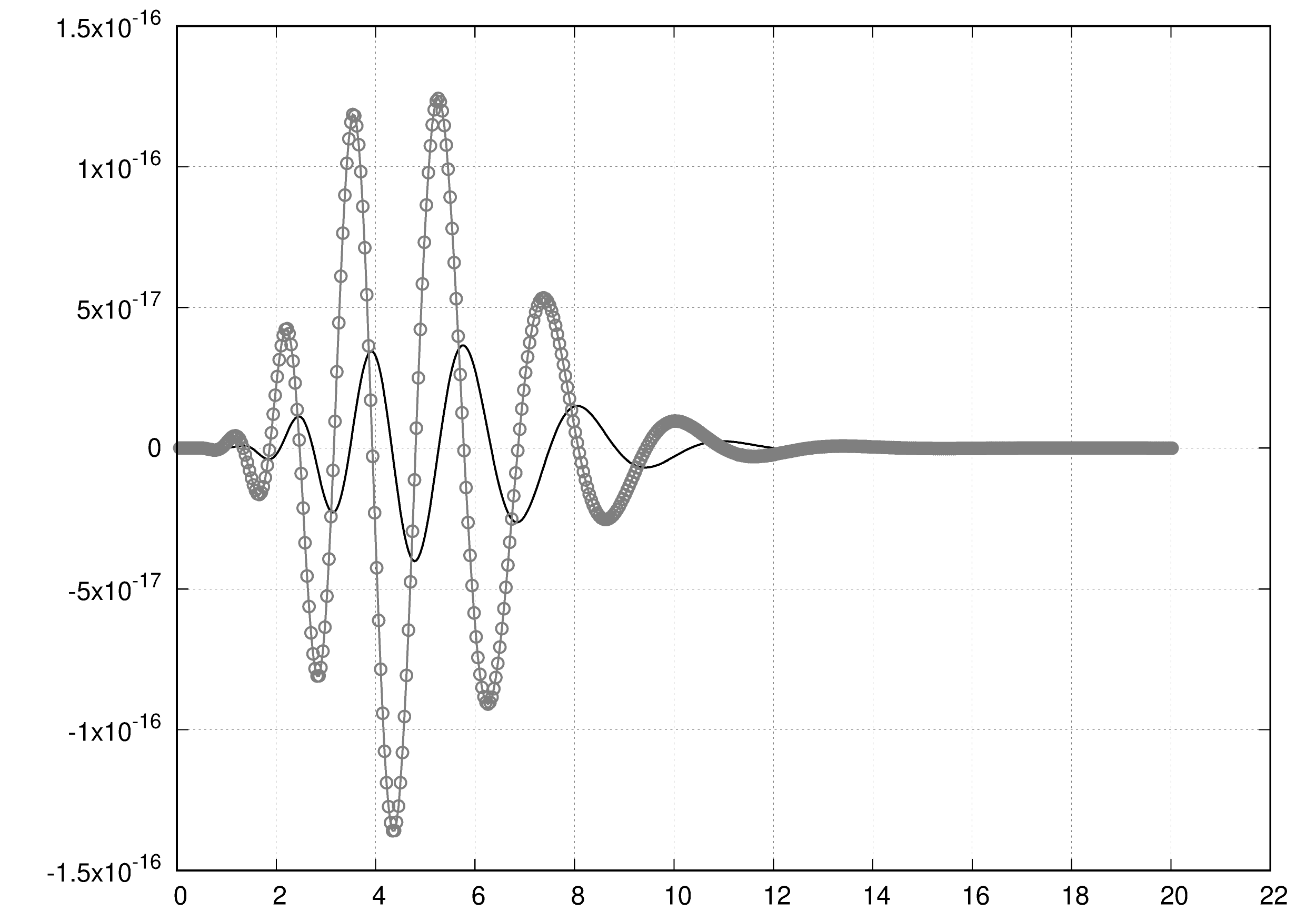}} 
\qquad
\subfloat[$N_1=55$, $N_2=35$ and $p=45$]{\label{fig2_p2}\includegraphics[height =6cm]{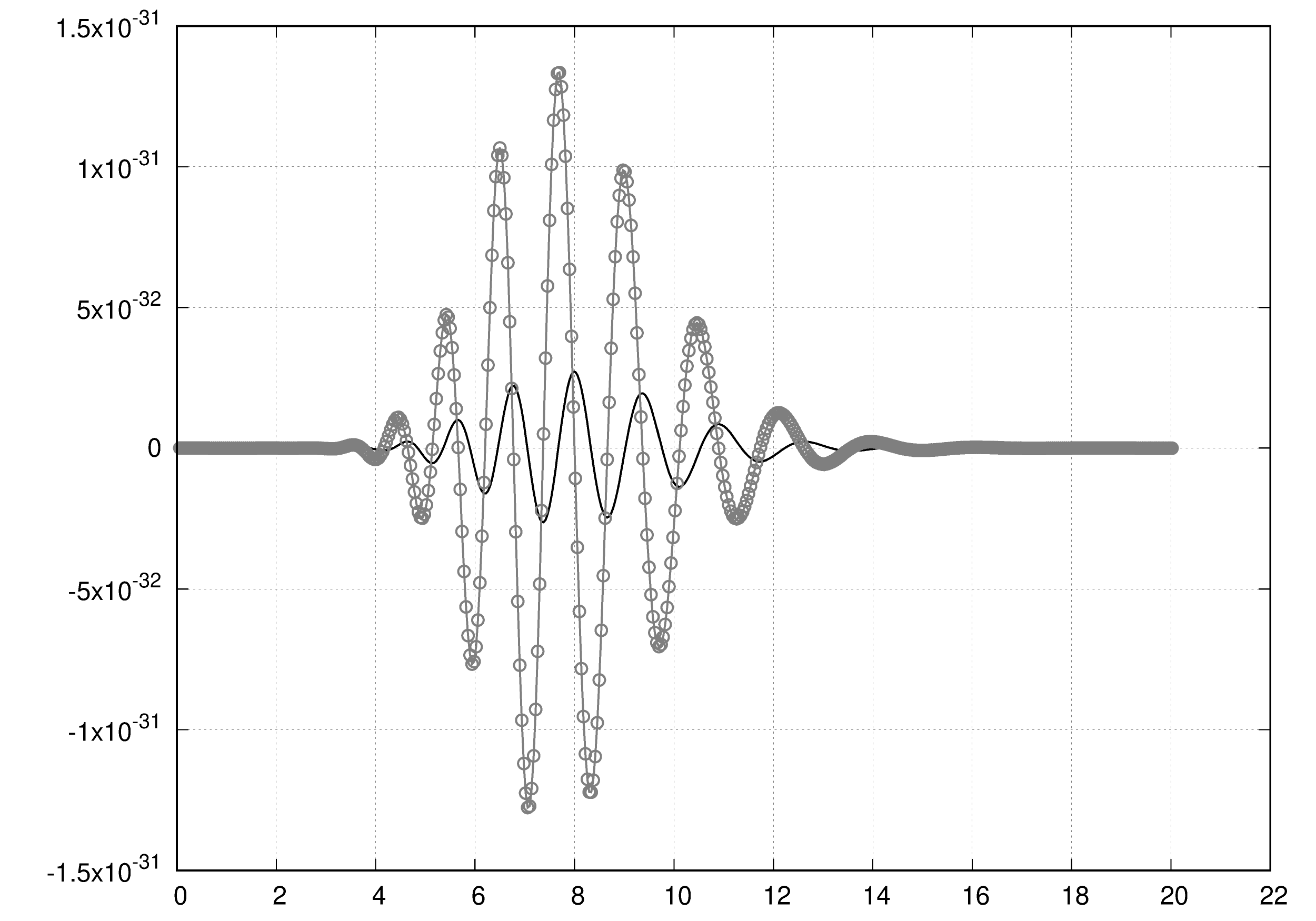}}
\caption{The graphs of the functions $\eta_1(x)$ (black line) and $\eta_2(x)$ (gray line with circles).} 
\label{fig2}
\end{figure}

Our next goal was to find values of $N_1$ and $N_2$ which generate approximations 
$\phi(x)= \sum\limits_{j=1}^p c_j e^{-\lambda_j x}$ to $f(x)$ such that the errors $\epsilon_1$ and $\epsilon_2$ defined in \eqref{eq:f_sum_error1} and 
\eqref{eq:f_sum_error2} are small.  We aimed to keep $p$ as small as possible (so that the number of terms required in the computation of $\Phi(z)$ in \eqref{def:Phi} is small) and to have the errors $\epsilon_1$ and $\epsilon_2$ in the range $(10^{-16}, 10^{-15})$, so that this would be adequate for double precision implementation of approximations to $\ln(\Gamma(z))$ and $\ln(G(z))$. In another experiment we aimed for the range $(10^{-31}, 10^{-30})$, that would be suitable for quadruple precision implementation of $\ln(\Gamma(z))$ and $\ln(G(z))$. By trial and error we found that taking $p=15$, $N_1=22$ and $N_2=8$ produces an approximation $\phi(x)$ with $\epsilon_1<10^{-16}$ and $\epsilon_2<3 \times 10^{-16}$. These upper bonds are not rigorous, rather they were established numerically by examining the graphs of 
\begin{align*}
\eta_1(x)&:=x^2 (f(x)-\phi(x)), \\
\eta_2(x)&:=3x(f(x)-\phi(x))+x^2(f'(x)- \phi'(x)).
\end{align*}
Note that, according to definitions \eqref{eq:f_sum_error1} and 
\eqref{eq:f_sum_error2}, we have 
$$
\epsilon_i= 2 \sup\limits_{x>0} |\eta_i(x)|.
$$
On Figure \ref{fig2_p1} the graph of $\eta_1(x)$ is shown as a solid black line and the graph of $\eta_2(x)$ as a gray line marked with circles.  Figure \ref{fig2_p1} provides convincing evidence that $|\eta_1(x)|<5 \times 10^{-17}$ for all values of $x$, which implies $\epsilon_1<10^{-16}$. The estimate $\epsilon_2<3 \times 10^{-16}$ is established in the same way by inspecting the graph of $\eta_2(x)$. The coefficients $\{\lambda_j\}_{1\le j \le 15}$ and $\{c_j\}_{1\le j \le 15}$ obtained by the two-point Pad\'e algorithm in this case 
$N_1=22$ and $N_2=8$ are presented in Appendix \ref{AppendixB} and can also be downloaded as a list in a text file from the author's \href{https://kuznetsov.mathstats.yorku.ca/code/}{webpage}. Using these coefficients we have implemented in Matlab the approximation algorithm for $\ln(G(z))$ that was described in Section \ref{Section2} -- this Matlab program can also be downloaded from the author's \href{https://kuznetsov.mathstats.yorku.ca/code/}{webpage}. We used this Matlab program  to create the surface plots of real and imaginary part of $\ln(G(z))$ that are shown on Figure \ref{fig3}.

Similarly, when $p=45$, $N_1=55$ and $N_2=35$ we present the graphs of $\eta_1(x)$ and $\eta_2(x)$ in Figure \ref{fig2_p2} and from these graphs we estimate that $\epsilon_1<10^{-31}$ and $\epsilon_2< 3\times 10^{-31}$. The coefficients $\{\lambda_j\}_{1\le j \le 45}$ and $\{c_j\}_{1\le j \le 45}$ corresponding to the case 
$N_1=55$ and $N_2=35$ can be found in the \href{http://arxiv.org/abs/2109.12061}{arXiv} version of the paper and can also be downloaded from the author's \href{https://kuznetsov.mathstats.yorku.ca/code/}{ webpage}. These coefficients may be used to implement quadruple precision approximations to the logarithm of the Barnes $G$-function, giving more than 30 correct digits in the half-plane $\re(z)\ge 3/2$.

\begin{figure}[t!]
\centering
\subfloat[$\re \ln(G(z))$]{\includegraphics[height =6cm]{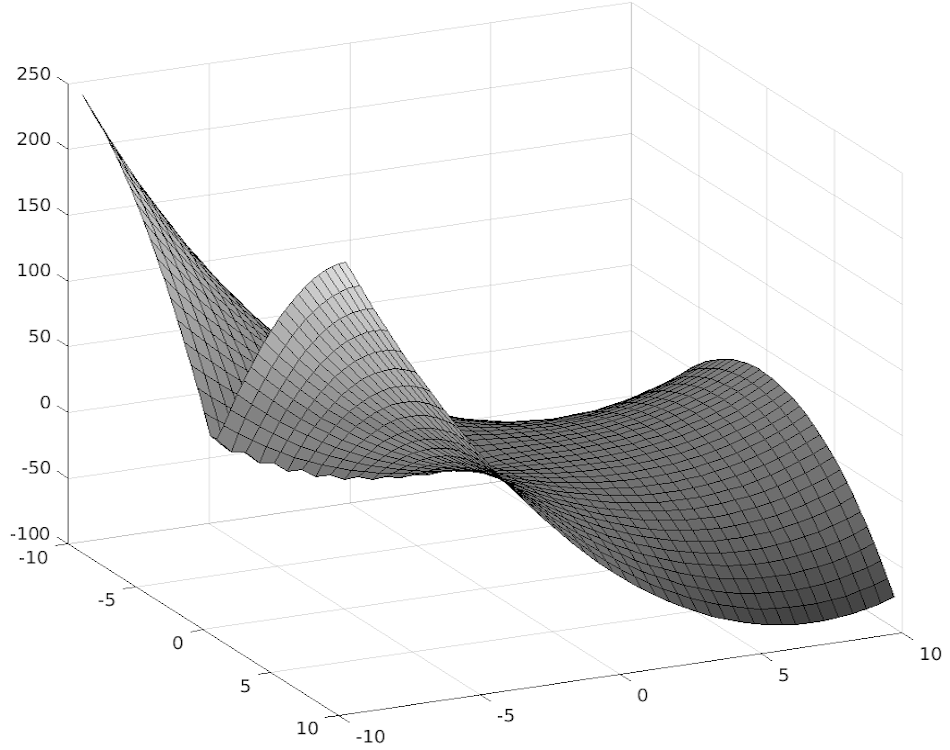}} 
\qquad
\subfloat[$\im \ln(G(z))$]{\includegraphics[height =6cm]{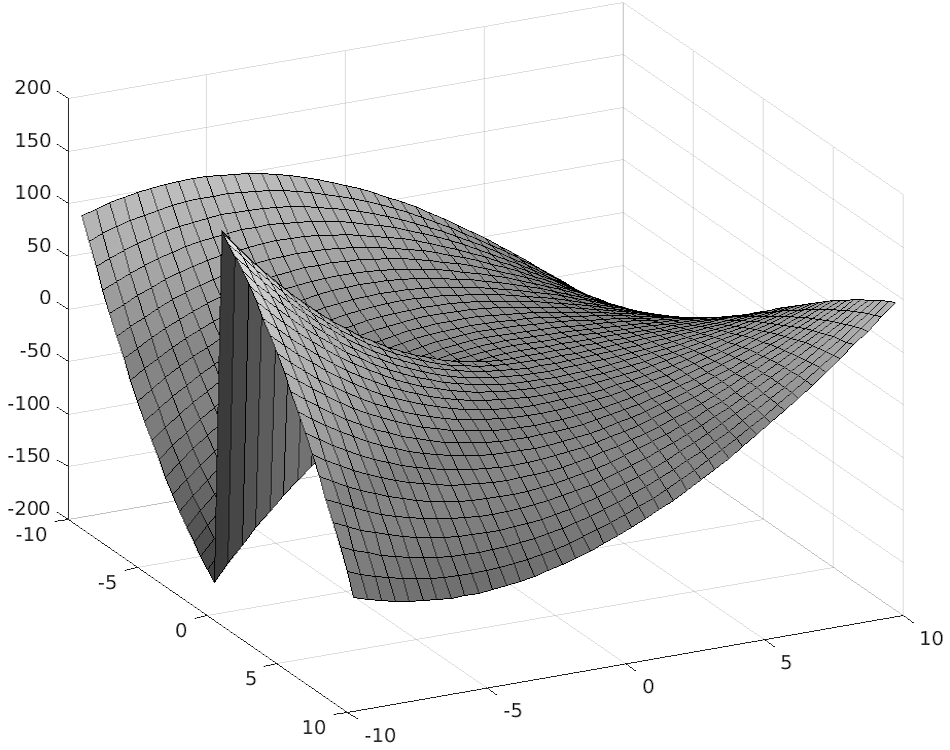}}
\caption{The surface plots of real and imaginary parts of $\ln(G(z))$.} 
\label{fig3}
\end{figure}

%

\section*{Acknowledgements}
Research was supported by the Natural Sciences and Engineering Research Council of Canada. The author would like to thank two anonymous referees for careful reading of the paper and for providing constructive comments and suggestions. 

%

\begin{appendices}

    \setcounter{proposition}{0}
    \renewcommand{\theproposition}{\Alph{section}\arabic{proposition}}
    \setcounter{theorem}{0}
    \renewcommand{\thetheorem}{\Alph{section}\arabic{theorem}}

\section{Computing the dilogarithm function for $|z|\le 1$}\label{AppendixA}
    \renewcommand{\theequation}{\Alph{section}.\arabic{equation}}
 
 We will follow closely the discussion in Section 4 in \cite{Maximon2003}. An efficient method for computing ${\textrm{Li}}_2(z)$ for $|z|\le 1$, which avoids the issue of 
 slow convergence of the series  
 \eqref{eq:def_Li2} when $|z|$ is close to one, can be obtained via the so-called Debye function 
 \begin{equation}\label{eq:def_Debye}
 D(w):=\int_0^w \frac{u \d u}{e^{u}-1}. 
 \end{equation}
 It is easy to see (by substitution $t=1-e^{-u}$) that 
 \begin{equation}\label{eq:Li2_as_D} 
 {\textrm{Li}}_2(z)=D(-\ln(1-z)).
\end{equation} 
  The Debye function has a Taylor series  expansion
 \begin{equation}\label{eq:Debye_Taylor}
 D(w)=w-\frac{w^2}{4}+\sum\limits_{n\ge 1} B_{2n} \frac{w^{2n+1}}{(2n+1)!},
 \end{equation} 
 where $B_n$ are Bernoulli numbers. From the asymptotics of Bernoulli numbers 
$$
|B_{2n}|=2(2n)! (2\pi)^{-2n}(1+o(1)), \;\;\; n\to +\infty
$$ 
 it is easy to deduce that the above Taylor series converges for $|w|< 2\pi$ (this fact also follows by noting that the integrand in 
 \eqref{eq:def_Debye} is analytic in the disk $|u|< 2\pi$, which implies that Debye function $D(w)$ must be analytic in $|w|<2\pi$). 
 
 Now, our algorithm for computing ${\textrm{Li}}_2(z)$ for $|z|\le 1$ consists of the following two steps:
 \begin{itemize}
 \item[(i)] If $z \in {\Omega_1}:=\{ z \in \c \; : \; |z| \le 1, \; \re(z)\le 1/2\}$ then we use identity \eqref{eq:Li2_as_D}  and compute $D(-\ln(1-z))$ via series expansion
 \eqref{eq:Debye_Taylor}.  One can check that for $z\in {\Omega}_1$ we have $|\ln(1-z)|\le \pi/3$, thus the terms in the series  \eqref{eq:Debye_Taylor} decay at the rate $O(6^{-2n})$. 
  \item[(ii)] If $z \in {\Omega_2}:=\{ z \in \c \; : \; |z|\le 1, \; \re(z)> 1/2\}$ then we first use identity (3.3) in \cite{Maximon2003}
  $$
  {\textrm{Li}}_2(z)=-{\textrm{Li}}_2(1-z)+\frac{\pi^2}{6}-\ln(1-z)\ln(z). 
  $$
 If $z\in \Omega_2$ then $1-z \in \Omega_1$, thus ${\textrm{Li}}_2(1-z)$ can be computed by the same approach as described in item (i) above. 
 \end{itemize}

 \section{Coefficients $\lambda_j$ and $c_j$ in the case $p=15$, $N_1=22$ and $N_2=8$}\label{AppendixB}
    \renewcommand{\theequation}{\Alph{section}.\arabic{equation}}

 Here we present values of $\{\lambda_j\}_{1\le j \le 15}$ and $\{c_j\}_{1\le j \le 15}$. The complex values come in conjugate pairs, thus we present only one such value and the other value may be computed by taking conjugate of the previous number (in other words,   $\lambda_{2n+1}=\overline{\lambda_{2n}}$ and $c_{2n+1}=\overline{c_{2n}}$ for 
 $4\le n \le 7$ ). These lists of numbers can also be downloaded as text files from the author's \href{https://kuznetsov.mathstats.yorku.ca/code/}{webpage}.

\begin{align*}
 \lambda_1&=1.015816941860969308,  \;\;\;
 \lambda_2=1.053963061918305102, \\
 \lambda_3&=1.116651540074509609, \;\;\;
 \lambda_4=1.207738507792217625, \\
 \lambda_5&=1.332888622825204091, \;\;\;
 \lambda_6=1.719941572880692604, \\
 \lambda_7&=2.930503690937967271, \\
 \lambda_8&=2.231464874614817990 - \i\times  0.280912039207008020, \\
  \lambda_{10}&=2.639898812086004465 - \i \times 0.873853916915943961, \\
 \lambda_{12}&=2.941124258312725471 - \i \times 1.605727317761697042, \\
 \lambda_{14}&=3.229198135526167105 + \i \times 2.596457178929701727, \\
c_1&=-3.361986110456561101\times 10^{-5}, \;\;\;
c_2=-1.894144561517152089\times 10^{-4}, \\
c_3&=-5.010483210821698243\times 10^{-4}, \;\;\;
c_4=-8.578556468220969250\times 10^{-4},  \\
c_5&=-8.943696088058549902\times 10^{-4}, \;\;\;
c_6=1.854241163038972664\times 10^{-3}, \\
c_7&=-1.918606889602829249\times 10^{-5}, \\
c_{8}&=-3.849191533344471619\times 10^{-4} + \i \times 2.988868248105834482\times 10^{-4}, \\
c_{10}&=1.121264751590328248\times 10^{-5} - \i \times 4.979727219667585924\times 10^{-6}, \\
c_{12}&=-1.113878636296735895\times 10^{-7} - \i \times 9.472403853117676266\times 10^{-8}, \\
c_{14}&=-1.508505417972961883\times 10^{-10}- \i \times 3.899201018438800852\times 10^{-10}.  \\
\end{align*}

 \end{appendices}


\end{document}